\documentclass[10pt]{article}
\usepackage{amsmath,amsfonts,amssymb,amsthm,graphicx}
\usepackage{enumerate}
\usepackage{etoolbox}
\usepackage{color}
\usepackage[colorlinks=true,citecolor=blue,pdfpagemode=UseNone,pdfstartview=FitH]{hyperref}

\usepackage[T2A]{fontenc}
\newenvironment{cyr}{}{}

\renewcommand{\d}{\,\mathrm{d}}

\newcommand{\R}{\mathbb{R}}

\newcommand{\e}{\mathrm{e}}

\DeclareMathOperator{\UXP}{\mathbb{P}^{\mathrm{X}}}
\DeclareMathOperator{\UXE}{\mathbb{E}^{\mathrm{X}}}
\DeclareMathOperator{\URP}{\mathbb{P}^{\mathrm{R}}}
\DeclareMathOperator{\URE}{\mathbb{E}^{\mathrm{R}}}

\DeclareMathOperator{\mX}{\mathit{m}^{\mathrm{X}}}
\DeclareMathOperator{\mR}{\mathit{m}^{\mathrm{R}}}

\newcommand{\NNN}{\mathcal{N}}

\theoremstyle{plain}
\newtheorem{theorem}{Theorem}

\newtheorem{lemma}[theorem]{Lemma}
\newtheorem{proposition}[theorem]{Proposition}
\theoremstyle{definition}

\theoremstyle{remark}
\newtheorem{remark}[theorem]{Remark}

\newcommand{\abstr}%
  {Randomness (in the sense of being generated in an IID fashion) and exchangeability
  are standard assumptions in nonparametric statistics and machine learning,
  and relations between them have been a popular topic of research.
  This short paper draws the reader's attention to the fact that,
  while for infinite sequences of observations the two assumptions are almost indistinguishable,
  the difference between them becomes very significant for finite sequences of a given length.}

\begin{document}
\title{Exchangeability and randomness for infinite and finite sequences}
\author{Vladimir Vovk}
\date{August 9, 2026}
\maketitle
\begin{abstract}
  \smallskip
  \abstr

  The version of this paper at \url{http://alrw.net} (Working Paper 45)
  is updated most often.
  The journal version is to appear in \emph{Statistical Science}.
\end{abstract}

\section{Introduction}

In this paper we will discuss two important assumptions
about sequences of observations,
exchangeability and randomness,
using the word ``randomness'' in a somewhat old-fashioned sense
of the individual observations being independent and identically distributed
(following \cite[Chap.~7]{Lehmann:1975},
which used the standard terminology for its time,
and \cite{Vovk/etal:2022book}).

The relationship between exchangeability and randomness
is very different in the cases of infinite sequences and finite sequences.
In the former case, there is hardly any difference between the two assumptions.
But in the latter, the difference may be vast.
The study of this relationship has a long history, which will also be briefly reviewed.

This paper was motivated by the fact
that the main property of validity of conformal prediction,
which has become a popular method for set prediction in recent years
\cite{Angelopoulos/etal:arXiv2411,Vovk/etal:2022book},
only depends on the assumption of exchangeability,
whereas the assumption that is usually made in mainstream machine learning
is that of randomness.
Aren't we losing a lot when we use conformal prediction,
which does not use the full power of the assumption of randomness?
This will be discussed on several occasions below.

We start in Sect.~\ref{sec:de-Finetti} with a discussion of de Finetti's theorem,
which was later greatly generalized by Hewitt and Savage and other people.
One implication of de Finetti's theorem is that,
for infinite sequences and for a wide range of observation spaces,
there is no difference between the assumptions of exchangeability and randomness.

In Sect.~\ref{sec:finite} we move on
to the case of finite sequences of observations of a given length.
The problem of the relation between exchangeability and randomness in this case
was implicitly posed by Kolmogorov \cite{Kolmogorov:1968-local}
in his work on the frequentist foundations of probability.
In the context of the algorithmic theory of randomness,
he simply defined randomness as exchangeability for binary sequences
(in which case the difference between the two assumptions is much less significant,
as discussed in Sect.~\ref{subsec:finite-Z}).
Precise difference for a natural alternative definition of randomness
was explored in work done under his supervision \cite{Vovk:1986-proofs}.
In Sect.~\ref{subsec:batch}
we will see a simple non-binary example
where the difference between exchangeability and randomness
is very substantial.

According to Kolmogorov,
and it is difficult to argue with his point of view,
infinite sequences are empirically vacuous;
we never observe them in reality.
Therefore, he always insisted on either studying finite sequences
or at least keeping them in mind.
In fact, de Finetti's results can also be stated in terms of finite sequences,
provided we consider sequences of different lengths.
In this paper we will briefly discuss two modes of statistical hypothesis testing
involving finite sequences, online and batch;
the former corresponds to de Finetti's setting and the latter to Kolmogorov's.

In Sect.~\ref{sec:tight} we check that the gap between exchangeability and randomness
described in Sect.~\ref{sec:finite} is the widest possible in some sense.
In its first part, Sect.~\ref{subsec:infinite-Z},
we discuss a simpler and cleaner mathematical result
that holds for an infinite observation space.
Then Sect.~\ref{subsec:finite-Z} is devoted to more complicated
and perhaps less practically relevant results about finite observation spaces.
However, these results cover Kolmogorov's binary case.
And we will see that exchangeability and randomness are asymptotically close in this case,
at least according to a relaxed standard proposed and sometimes used by Kolmogorov.

Section~\ref{sec:conclusion} concludes
and Appendices~\ref{app:A} and~\ref{app:C}
contain two short but technical proofs.

This paper might have little novelty,
and all of its mathematical results are very simple.
But the reader might find the scale of the difference
between exchangeability and randomness for infinite observation spaces
surprising.

\section{De Finetti's theorem}
\label{sec:de-Finetti}

Suppose we observe a data sequence $z_1,z_2,\dots\in\mathbf{Z}$
consisting of observations $z_n\in\mathbf{Z}$
that are elements of some measurable space $\mathbf{Z}$,
the \emph{observation space}.
The \emph{assumption of randomness} is standard in machine learning:
the data is coming from the product probability measure $Q^{\infty}$
for some $Q\in\mathfrak{P}(\mathbf{Z})$,
where $\mathfrak{P}(Z)$ stands for the measurable space
of all probability measures on a measurable space $Z$
(with the $\sigma$-algebra on $\mathfrak{P}(Z)$
generated by the mappings $Q\mapsto Q(A)$,
$A$ being an event in $Z$).

\begin{remark}
  Effects of various distribution shifts have also been widely studied,
  but in this paper we concentrate on the basic IID case.
  In machine learning we often have observations $z=(x,y)$
  consisting of an object $x$ and its label $y$,
  but we do not insist on this.
\end{remark}

The more general \emph{assumption of exchangeability} is that
the data is coming from an \emph{exchangeable} probability measure
$R\in\mathfrak{P}(\mathbf{Z}^{\infty})$,
i.e., a probability measure that is invariant w.r.t.\ swapping any pair of observations.
The topic of this section is the closeness of the two assumptions
for infinite, or at least potentially infinite, data sequences
and assuming that $\mathbf{Z}$ is a \emph{Borel space},
i.e., a measurable space that is isomorphic to a Borel subset of $\R$.

By de Finetti's classical theorem
each exchangeable probability measure $R$ on $\R^{\infty}$
is a mixture of product distributions:
\begin{equation}\label{eq:representation}
  R = \int Q^{\infty} \mu(\d Q)
\end{equation}
for some $\mu\in\mathfrak{P}(\mathfrak{P}(\R))$.
This was established by de Finetti \cite[Chap.~4]{deFinetti:1937-local}.
Hewitt and Savage \cite[Theorem~7.3 and its discussion later in Sect.~7]{Hewitt/Savage:1955}
note that we can trivially replace $\R$ by any Borel space $\mathbf{Z}$
and point out that, to the best of their knowledge,
every measurable space known to have importance in applied science is Borel.
For example, every Polish space
(complete separable metric space with its Borel $\sigma$-algebra) is Borel.
On the other hand, there exists a separable metric space with its Borel $\sigma$-algebra
for which \eqref{eq:representation} can be violated \cite{Dubins/Freedman:1979}.
Let us assume in the rest of this paper that the observation space $\mathbf{Z}$ is Borel.

It is well known that de Finetti's theorem fails
if we simply replace $\infty$ by a finite $N$ in \eqref{eq:representation};
see, e.g., \cite[Sect.~4.7.1]{Bernardo/Smith:2000}
(and \cite[Sect.~1]{Diaconis:1977} for a further discussion
of the extent to which de Finetti's theorem can fail
for $N=2$ and $\mathbf{Z}=\{0,1\}$).

\begin{remark}\label{rem:anonymous}
  An anonymous reviewer pointed out that,
  in the case of a finite $N$,
  Shtarkov's normalized maximum likelihood distribution
  (see, e.g., \cite[Sect.~6.2.1]{Grunwald:2007})
  is a prime example of an exchangeable probability measure
  that, in general, cannot be represented in the form \eqref{eq:representation}.
  But perhaps the simplest such example is the uniform probability measure
  on the set of all permutations of a data sequence with distinct elements.
\end{remark}

De Finetti's theorem plays an important role in the foundations of Bayesian statistics;
see, e.g., \cite[Chap.~4]{Bernardo/Smith:2000}.
But its implication in our context is that
the assumptions of exchangeability and randomness are equivalent:
if a testing procedure rejects one of these assumptions,
it rejects the other as well.
We will formalize this statement in three different ways.

The simplest way of statistical hypothesis testing is based on critical regions.
To test a composite null hypothesis $\mathcal{H}$
(a set of probability measures)
at a significance level $\epsilon\in(0,1)$ (such as $1\%$ or $5\%$),
we choose a \emph{critical region} $A$ at level $\epsilon$,
meaning an event satisfying $R(A)\le\epsilon$ for all $R\in\mathcal{H}$.
The hypothesis $\mathcal{H}$ is rejected if we observe $A$,
which must be chosen in advance.
Applying this to the assumption of exchangeability,
critical regions $A\subseteq\mathbf{Z}^{\infty}$ for testing exchangeability
at level $\epsilon$
are required to satisfy $\UXP(A)\le\epsilon$,
where the \emph{upper exchangeability probability} of $A$ is defined by
\begin{equation}\label{eq:UXP}
  \UXP(A)
  :=
  \sup_{R}
  R(A),
\end{equation}
$R$ ranging over the exchangeable probability measures on $\mathbf{Z}^{\infty}$.
Similarly,
critical regions $A\subseteq\mathbf{Z}^{\infty}$ for testing randomness
at level $\epsilon$
are required to satisfy $\URP(A)\le\epsilon$,
where the \emph{upper randomness probability} of $A$ is defined as
\begin{equation}\label{eq:URP}
  \URP(A)
  :=
  \sup_{Q\in\mathfrak{P}(\mathbf{Z})}
  Q^{\infty}(A).
\end{equation}
By de Finetti's theorem, $\UXP$ and $\URP$ coincide:
\begin{itemize}
\item
  since every product measure $Q^{\infty}$ is exchangeable,
  $\URP\le\UXP$;
\item
  on the other hand, since each exchangeable probability measure $R$
  has a representation~\eqref{eq:representation},
  we have, for each $\alpha>0$ and each event $A$ satisfying $\URP(A)\le\alpha$,
  \begin{equation}\label{eq:UXP_vs_URP}
    \UXP(A)
    =
    \sup_{R}
    R(A)
    \le
    \sup_{\mu}
    \int Q^{\infty}(A) \mu(\d Q)
    \le
    \alpha;
  \end{equation}
  therefore, $\UXP\le\URP$.
\end{itemize}
We can see that there are exactly the same critical regions
at each significance level
under exchangeability and under randomness.

One manifestation of the coincidence of the critical
regions under exchangeability and randomness is that the two assumptions
will produce identical prediction sets
\begin{equation}\label{eq:Gamma}
  \Gamma(z_1,\dots,z_n)
  :=
  \bigl\{
    (z_{n+1},z_{n+2},\dots):
    (z_1,\dots,z_n,z_{n+1},z_{n+2},\dots)\notin A
  \bigr\}
\end{equation}
for the future observations after observing $z_1,\dots,z_n$,
where $A$ is a critical region at some significance level $\epsilon$.
Under both exchangeability and randomness,
the coverage probability of the prediction set \eqref{eq:Gamma}
will be at least $1-\epsilon$.

There are two popular generalizations of critical regions,
p-variables and e-variables,
and both also produce identical results under exchangeability and randomness.
Let us first check this for p-variables.
According to the general definition
(see, e.g., \cite[Definition 1.2]{Ramdas/Wang:2025}),
an \emph{exchangeability p-variable} is a random variable
$P:\mathbf{Z}^{\infty}\to[0,1]$
such that, for all $\epsilon\in(0,1)$,
\begin{equation}\label{eq:X-p}
  \UXP(P\le\epsilon)
  \le
  \epsilon.
\end{equation}
Similarly,
a \emph{randomness p-variable} is a random variable
$P:\mathbf{Z}^{\infty}\to[0,1]$
such that, for all $\epsilon\in(0,1)$,
\begin{equation}\label{eq:R-p}
  \URP(P\le\epsilon)
  \le
  \epsilon.
\end{equation}
Since $\UXP=\URP$,
the classes of exchangeability and randomness p-variables coincide.

Again according to the general definition
\cite[Definition 1.2]{Ramdas/Wang:2025}
(see also \cite[(1)]{Grunwald/etal:2024}),
an \emph{exchangeability e-variable} is a measurable function
$F:\mathbf{Z}^{\infty}\to[0,\infty]$
such that $\UXE(F)\le1$, where
\begin{equation}\label{eq:UXE}
  \UXE(F)
  :=
  \sup_R
  \int F \d R,
\end{equation}
$R$ ranging over the exchangeable probability measures on $\mathbf{Z}^{\infty}$.
And a \emph{randomness e-variable} is a measurable function
$F:\mathbf{Z}^{\infty}\to[0,\infty]$
such that $\URE(F)\le1$, where
\begin{equation}\label{eq:URE}
  \URE(F)
  :=
  \sup_{Q\in\mathfrak{P}(\mathbf{Z})}
  \int F \d Q^{\infty}.
\end{equation}
Let us check that $\UXE=\URE$.
Since $\URE\le\UXE$ is obvious, we just need to check $\UXE\le\URE$.
Generalizing \eqref{eq:UXP_vs_URP},
we have, for each $\alpha>0$,
each measurable function $F:\mathbf{Z}^{\infty}\to[0,\infty]$
satisfying $\URE(F)\le\alpha$,
and each $\delta>0$,
\begin{equation}\label{eq:UXE_vs_URE}
  \UXE(F)
  \le
  \int F \d R + \delta
  =
  \int\!\!\int F \d Q^{\infty} \mu(\d Q) + \delta
  \le
  \alpha + \delta,
\end{equation}
and so indeed $\UXE\le\URE$.
The first inequality in \eqref{eq:UXE_vs_URE} holds
for some exchangeable probability measure $R$
and the second for some $\mu\in\mathfrak{P}(\mathfrak{P}(\mathbf{Z}))$
according to \eqref{eq:representation}.
As $\UXE=\URE$,
the classes of exchangeability and randomness e-variables coincide.

Whatever testing method out of the three we use,
we have the same options for rejecting exchangeability or randomness.
The empirical contents of the two assumptions may be said to coincide,
under our assumption of $\mathbf{Z}$ being a Borel space.

\section{Finite sequences of observations}
\label{sec:finite}

In the previous section we saw that, by de Finetti's theorem,
the difference between the assumptions of exchangeability and randomness disappears.
However, this is a statement about infinite sequences,
which we can never observe in reality.
Can we say something similar for finite sequences of observations?

\subsection{Batch setting}
\label{subsec:batch}

We start our discussion of finite sequences of observations
from the simplest setting in which we fix the number $N$ of observations
and only consider the sequences of length $N$, $\mathbf{Z}^N$.
We call this the \emph{batch setting}.
In this case a chasm between exchangeability and randomness opens up.
Consider the following very simple critical region $A$.
The $N$ observations are all in the set $\{1,\dots,N\}$
and are all different.
Under exchangeability, the event $A$ is perfectly possible:
its probability is 1 under some exchangeable probability measure.
Defining $\UXP$, $\URP$, $\UXE$, and $\URE$ as in the previous section
(see \eqref{eq:UXP}, \eqref{eq:URP}, \eqref{eq:UXE}, and \eqref{eq:URE})
but replacing $\mathbf{Z}^{\infty}$ with $\mathbf{Z}^N$
and $Q^{\infty}$ with $Q^N$,
we can see that $\UXP(A)=1$.
The maximum probability of $A$ under any product measure $Q^N$ is
\begin{equation}\label{eq:product}
  \frac{N}{N}
  \frac{N-1}{N}
  \dots
  \frac{1}{N}
  =
  \frac{N!}{N^N}
  \sim
  \sqrt{2\pi N}
  \e^{-N};
\end{equation}
therefore, it shrinks exponentially fast as $N$ grows.
Indeed, the maximum of $Q^N(A)$ is achieved
for the $Q$ concentrated on $\{1,\dots,N\}$
and uniformly distributed on this set.
Applying Stirling's formula in the form of \cite{Robbins:1955}
gives
\begin{equation}\label{eq:Stirling}
  \URP(A) < 3\sqrt{N}\e^{-N}
  <
  1 = \UXP(A)
\end{equation}
(the second inequality ``$<$'' assumes $N>1$).
For example, suppose we are interested in significance level $\epsilon:=10^{-k}$ for $k\ge2$
(such as $k=2$ for high statistical significance).
Solving $3\sqrt{N}\e^{-N}\le10^{-k}$,
we obtain
\[
  \URP(A) < 10^{-k}
  <
  1 = \UXP(A)
\]
provided $N\ge3k+1$.
For example, an outcome that is perfectly possible under exchangeability ($\UXP(A)=1$)
becomes highly statistically significant under randomness for $N=7$.
The much stricter significance level of ``5 sigma'',
approximately $1/(3\times10^6)$,
used for announcing discoveries in particle physics \cite{ATLAS:2012}
is met starting from $N=3\times7+1=22$
(in fact already from $N=18$ according to \eqref{eq:Stirling}).

\begin{remark}
  A convenient, albeit asymptotically much cruder,
  version of the inequality~\eqref{eq:Stirling} is
  \begin{equation*}
    \URP(A) \le 2^{-N+1}
    \le
    1 = \UXP(A).
  \end{equation*}
\end{remark}

\begin{remark}
  The difference between exchangeability and randomness
  is also manifested by the fact that exchangeability
  is easy to achieve in practice:
  we can just permute randomly our data sequence.
  However, the resulting sequence may be very far from being IID.
\end{remark}

Similarly to~\eqref{eq:Gamma},
we can invert exchangeability and randomness critical regions
in the batch mode
and use them for one-step-ahead prediction
(whereas prediction in~\eqref{eq:Gamma} was infinitely many steps ahead).
Given an observed sequence $z_1,\dots,z_n$,
we output
\begin{equation*}
  \Gamma(z_1,\dots,z_n)
  :=
  \left\{
    z_{n+1}:
    (z_1,z_2,\dots,z_{n+1})\notin A
  \right\}
\end{equation*}
as our prediction set for the next observation,
where $A$ is a critical region in $\mathbf{Z}^N$ with $N:=n+1$;
this ensures a coverage probability of at least $1-\epsilon$,
where $\epsilon$ is the significance level used in $A$.
Let $A$ be as in \eqref{eq:Stirling}
(all permutations of $\{1,\dots,N\}$).
Under randomness, the corresponding set predictor
can produce non-vacuous
(i.e., different from $\mathbf{Z}$) prediction sets
at significance level $3\sqrt{N}\e^{-N}$
while it is incapable of producing non-vacuous prediction sets
under exchangeability
at any non-trivial significance level
(i.e., significance level below $1$).

It is instructive to compare
the exponentially small significance level $3\sqrt{N}\e^{-N}$
with what can be achieved in conformal prediction.
Conformal prediction associates
with each possible value of the test observation $z_N$
a p-value equal to the appropriately computed rank of $z_N$
in the multiset $\{z_1,\dots,z_N\}$
divided by $N$ \cite[Sect.~11.4.4]{Vovk/etal:2022book}.
This implies that non-vacuous prediction sets are only possible
at significance levels of at least $1/N$,
which is incomparably larger than $3\sqrt{N}\e^{-N}$.
What is the origin of this weakness of conformal prediction?
(Or apparent weakness, as I will argue momentarily.)

Given a target probability of error $\epsilon>0$,
we define a prediction set for $z_N$ by including in it
only the potential values of $z_N$ that lead to p-values exceeding $\epsilon$.
If $\epsilon N$ is an integer,
then the probability of error will be indeed $\epsilon$
under mild conditions
(or if we use ``smoothed p-values'' \cite[Sect.~2.2.6]{Vovk/etal:2022book},
without any conditions).
An error here means the prediction set failing to cover the true $z_N$.
This is the main property of validity of conformal prediction,
and it only requires the exchangeability of $z_1,\dots,z_N$,
not their randomness.
It is interesting that conformal predictors can even be defined
as the set predictors for $z_N$ that are valid under exchangeability
and are \emph{train-invariant},
in the sense of not depending on the order of $z_1,\dots,z_n$
\cite[Proposition~1]{Nouretdinov/etal:2003ALT}.
The requirement of train-invariance is very natural
and follows
(as mentioned in \cite[Sect.~6.2]{Vovk:2026Gam})
from two basic principles of statistical inference,
the sufficiency principle and the invariance principle
(see, e.g., \cite[Sect.~2.3, (ii) and (vi)]{Cox/Hinkley:1974}).
The main reason for the weakness of conformal prediction
pointed out in the previous paragraph
is its validity under exchangeability rather than under randomness only
(the set $A$ discussed earlier in this subsection
is train-invariant).

However, the weakness is to a large degree apparent,
as explained in \cite[Theorem 2]{Nouretdinov/etal:2003ALT}
and made more precise, getting rid of unspecified constants,
in \cite[Theorem~5]{Vovk:2026ALT} and \cite{Vovk:2026Gam}.
For example, the results of \cite{Vovk:2026ALT}
(Theorem~5 and its corollaries, such as Corollary~14)
can be crudely summarized as follows.
Every set predictor $\Gamma$ that is valid in the same sense as conformal predictors
but under randomness instead of exchangeability
can be turned into a conformal predictor $\Gamma'$ and a randomness e-variable $F$
such that $\Gamma'$ produces a similar prediction set to $\Gamma$
unless $F$ takes a large value.
Therefore, typically $\Gamma$ and $\Gamma'$ produce similar prediction sets.
Intuitively, the chasm between randomness and exchangeability
as bases for prediction
only shows for data sequences that are untypical under randomness.

\subsection{Kolmogorov's and Martin-L\"of's work on Bernoulli sequences}
\label{subsec:Bernoulli}

The assumptions of randomness and exchangeability,
in different guises,
have also played important roles in the foundations of frequentist statistics.
The standard measure-theoretic foundations of probability were put forward
in Kolmogorov's \emph{Grundbegriffe} \cite{Kolmogorov:1933},
but Kolmogorov did not believe that they were sufficient
for applications of probability.
As he pointed out in the \emph{Grundbegriffe} \cite[footnote 4]{Kolmogorov:1933},
in his frequentist analysis of the applications of probability
he was following Richard von Mises.
However, a big difference between Kolmogorov's and von Mises's approaches
was that von Mises's was based on infinite sequences,
whereas Kolmogorov believed that infinite sequences,
being empirically non-existent,
had no place in discussions of real-world applications of probability.
Interestingly, because of this he even objected
against publication in \emph{Russian Mathematical Surveys}
(a journal that he edited at the time)
of a planned paper about infinite random sequences
by his student Uspensky and close collaborators Shen and Semenov;
see Kolmogorov's letter to Uspensky of June 1983
cited in \cite[note 14]{Semenov/etal:2024-full}.

For a long time Kolmogorov believed that no frequentist concept of probability
can be developed for finite sequences \cite[Sect.~1]{Kolmogorov:1963},
but in 1963 he published his first attempt in this direction
\cite[Sect.~2]{Kolmogorov:1963}.
The attempt, however, was ``incomplete''
(as he characterizes it in \cite[Sect.~4]{Kolmogorov:1965-local}),
and he greatly improved on it in 1968 \cite[Sect.~2]{Kolmogorov:1968-local}
(this paper is based on his 1967 talk).
It appears that the details of Kolmogorov's improved approach
first appeared in print in Martin-L\"of's 1966 paper \cite[Sect.~5]{Martin-Lof:1966}.

Both Kolmogorov and Martin-L\"of consider binary sequences
and define what they call Bernoulli sequences,
i.e., sequences that can be plausibly obtained as a result of IID observations.
While their informal explanations clearly show
that they are interested in the assumption of randomness,\footnote{%
  This is a relevant quote from Kolmogorov \cite[Sect.~2]{Kolmogorov:1968-local}:
  ``let us consider how we imagine a sequence of zeros and ones
  appearing as the result of independent trials
  with probability $p$ of obtaining a one at each trial.''}
their formal definitions are about the assumption of exchangeability.
Let me give essentially Martin-L\"of's definition;
I will use slightly different terminology,
but my definition will be equivalent to Martin-L\"of's.
This definition will rely on some basic notions of the theory of algorithms
(see, e.g., \cite[Sect.~1.7]{Li/Vitanyi:2019}),
but it will not be used outside this subsection,
and the reader can skip the rest of the subsection
without interrupting the flow of ideas.

Let us define exchangeability p-variables and randomness p-variables
as in the previous section, by \eqref{eq:X-p} and \eqref{eq:R-p},
but replacing $\mathbf{Z}^{\infty}$ with $\mathbf{Z}^N$ for a finite $N$;
at first we assume $\mathbf{Z}=\{0,1\}$.
We consider families $P_N$, $N\in\{1,2,\dots\}$, of exchangeability p-variables
such that $P_N(x)$ is upper semicomputable
as function of $N$ and $x\in\mathbf{Z}^N$
(where the upper semicomputability means that,
for a computable function $f$ taking rational values,
$P_N(x)=\inf_k f(N,x,k)$, $k$ ranging over the natural numbers).
There exists a smallest,
to within a constant factor,
function $(N,x)\mapsto P_N(x)$ of this kind,
and Martin-L\"of defines
\[
  m(x)
  :=
  -\log_2 P_N(x)
\]
for all binary sequences $x\in\{0,1\}^*$,
where $N$ is the length of $x$ and $\log_b$ is base-$b$ logarithm.
This definition is slightly arbitrary,
since $m(x)$ is only defined to within an additive constant,
so additive terms of $O(1)$ are typically ignored
in the algorithmic theory of randomness.

Another equivalent definition of $m$ is given
by Kolmogorov in \cite[Sect.~2]{Kolmogorov:1968-local}
in terms of his notion of complexity
(and Martin-L\"of proves the equivalence in \cite[Sect.~5]{Martin-Lof:1966}).
A binary sequence $x$ is called a \emph{Bernoulli sequence}
if $m(x)$ is small;
this is an informal notion, but we can prove mathematical results
about the function $m$ (albeit only with the $O(1)$ accuracy).

Replacing the assumption of exchangeability by that of randomness,
we get a function that we denote by $\mR$ instead of $m$.
(A natural notation for $m$ in view of our notations $\UXP$ and $\UXE$
would have been $\mX$,
but $m$ is what Martin-L\"of used in \cite[Sect.~5]{Martin-Lof:1966}.)

These definitions can be adapted verbatim to the case
where the observation space $\mathbf{Z}$ is $\{1,2,\dots\}$
rather than $\{0,1\}$.
In this case the example in Sect.~\ref{subsec:batch} demonstrating \eqref{eq:Stirling}
shows that there are sequences $x_N\in\mathbf{Z}^N$, $N=1,2,\dots$,
such that
\begin{equation*}
  m(x_N) = O(1)
  <
  N\log_2\e - \frac12\log_2N + O(1)
  =
  \mR(x_N),
\end{equation*}
the inequality holding from some $N$ on;
this is stated without proof in \cite[Theorem~4]{Vovk/etal:1999-local}.
We can see that the difference between $m$ and $\mR$ is very substantial.

The algorithmic approach to exchangeability and randomness
was used in \cite{Nouretdinov/etal:2003ALT} for exploring conformal prediction.
This was the source of unspecified constants in its results,
mentioned at the end of Sect.~\ref{subsec:batch}.

\subsection{Versions of de Finetti's theorem for finite sequences of observations}

De Finetti's theorem is about infinite sequences of observations,
similarly to von Mises's frequentist story criticised by Kolmogorov.
Does it mean that de Finetti's theorem is empirically irrelevant?
Not at all.
Whereas von Mises's story may be hopelessly stuck at infinity
(to use Shafer's expression \cite{Shafer:1993}),
de Finetti's theorem can be applied to finite sequences,
albeit not in the batch setting of Sect.~\ref{subsec:batch}.
A popular alternative to the batch setting is the \emph{online setting},
where we do not fix the number of observations in advance
and process them sequentially,
and then the sequence of observations becomes potentially infinite;
therefore, de Finetti's theorem becomes applicable
if we assume exchangeability for all those potential observations.
For example, in Sect.~6 of \cite{Martin-Lof:1966}
Martin-L\"of develops a way of testing exchangeability
for all finite prefixes of a potentially infinite sequence of observations,
which means that he is indeed testing randomness of the overall sequence.

The main property of validity of conformal prediction
is greatly strengthened in the online setting.
Namely, in the online mode of prediction
conformal predictors not only make errors with prespecified probabilities,
but they also make errors at different steps independently
(see, e.g., \cite[Sect.~8.1]{Angelopoulos/etal:arXiv2411}
or \cite[Theorem~11.1]{Vovk/etal:2022book}).
Because of this, the law of large numbers,
large deviation inequalities, etc., become applicable,
and the target error probability translates into the long-run frequency of errors.
This, however, requires exchangeability for each prefix of the data sequence
(potentially infinite),
i.e., by de Finetti's theorem, requires randomness.

A useful finite form of de Finetti's theorem was derived
by Diaconis and Freedman \cite{Diaconis/Freedman:1980finite}
(with an early version given already at the very end of \cite{Hewitt/Savage:1955}).
Since the assumptions of exchangeability and randomness are so different
for a fixed length $N$,
we have to consider different lengths for imposing exchangeability
and for claiming randomness.
In the abstract of \cite{Diaconis:1977},
which states the Diaconis--Freedman result for binary sequences, Diaconis
summarizes this result thus:
``an exchangeable sequence of length $r$
which can be extended to an exchangeable sequence of length $k$
is almost a mixture of independent experiments,
the error going to zero like $1/k$''.
It is essential that $k$ and $r$ should be different here,
ideally $k\gg r$.
There have been several recent information-theoretic developments of this idea;
see \cite[Corollary~1]{Johnson/etal:2025} for a particularly strong result.

Another popular finite version of de Finetti's theorem
appears in Dellacherie and Meyer \cite[Chap.~5, 52]{Dellacherie/Meyer:1982},
who credit this result to P.~Cartier;
see Kerns and Sz\'ekely \cite{Szekely/Kerns:2006} for a fuller exposition.
Theorem~1.1 in \cite{Szekely/Kerns:2006} says
that the representation \eqref{eq:representation}
holds in the batch setting for any exchangeable probability measure $R$ on $\mathbf{Z}^N$
without any restrictions on the measurable space $\mathbf{Z}$
if we allow $\mu$ to be a signed measure of bounded variation.
This result appears to be a mathematical curiosity
that does not have any implications for statistical hypothesis testing.
To derive interesting implications,
we need $\mu$ to be a probability measure,
and the existence of such a $\mu$ is sometimes postulated
even for finite sequences of a given length;
see, e.g., \cite[Sect.~7, Proposition]{Banerjee/Richardson:2013}.

Finally, ``de Finetti's theorem'' is sometimes used in a much wider sense
covering representations of exchangeable probability measures
as mixtures of probability measures different from product measures,
as in \cite[Theorem~1]{Diaconis:1977}.
We do not discuss such results as our main interest
is relations between exchangeability and randomness.

\section{Tight inequalities}
\label{sec:tight}

We first consider the simpler case of an infinite observation space $\mathbf{Z}$
and then move on to the slightly messier case of a finite $\mathbf{Z}$.
We always assume that all singletons in $\mathbf{Z}$ are measurable;
among other things, this will ensure that the effective size of a finite $\mathbf{Z}$
is equal to its cardinality $\left|\mathbf{Z}\right|$.

\subsection{The case of infinite $\mathbf{Z}$}
\label{subsec:infinite-Z}

In this subsection we will check that the example given in the previous section
(see \eqref{eq:product} and \eqref{eq:Stirling})
is the most extreme when the observation space $\mathbf{Z}$ is infinite.
Namely, we will see that,
for any $\mathbf{Z}$ (finite or infinite) and any event $A$ in $\mathbf{Z}^N$,
\begin{equation}\label{eq:bound}
  \UXP(A) \le \frac{N^N}{N!} \URP(A);
\end{equation}
the example demonstrates that the equality here is attained
for a large $\left|\mathbf{Z}\right|$.
The length $N$ is fixed throughout this subsection.

The inequality \eqref{eq:bound} can be strengthened to the following proposition.

\begin{proposition}\label{prop:infinite-Z}
  For any random variable $F:\mathbf{Z}^N\to[0,\infty)$,
  \begin{equation}\label{eq:to-prove}
    \UXE(F) \le \frac{N^N}{N!} \URE(F).
  \end{equation}
  Suppose $\left|\mathbf{Z}\right|\ge N$.
  The bound is tight and attained on a non-zero indicator function.
  Moreover,
  there exists an event $A$ in $\mathbf{Z}^N$
  such that
  \[
    \URP(A)
    =
    \frac{N!}{N^N}
    \le
    1 = \UXP(A).
  \]
\end{proposition}

\begin{proof}
  To prove \eqref{eq:to-prove},
  first notice that the left-hand side
  is the supremum of the averages of $F$ over all orbits,
  where an \emph{orbit} is defined to be
  the set of all sequences obtained by permuting
  a given sequence in $\mathbf{Z}^N$;
  as a formula,
  \[
    \UXE(F)
    =
    \sup_{z_1,\dots,z_N}
    \bar F(z_1,\dots,z_N),
  \]
  where
  \[
    \bar F(z_1,\dots,z_N)
    :=
    \frac{1}{N!}
    \sum_{\pi\in S_N}
    F(z_{\pi(1)},\dots,z_{\pi(N)})
  \]
  and $S_N$ stands for the symmetric group of all permutations of $\{1,\dots,N\}$
  (this follows from, e.g., \cite[Lemma~A.3]{Vovk/etal:2022book}).
  Therefore, it suffices to consider only $F$ that are non-zero on one orbit only.
  Let $F$ be non-zero on the orbit generated
  by a sequence $z_1,\dots,z_N$ in $\mathbf{Z}^N$
  containing $K$ distinct elements of $\mathbf{Z}$
  with multiplicities $n_1,\dots,n_K$ (so that $n_1+\dots+n_K=N$).
  The largest product probability $Q^N$ of an element of this orbit is
  \[
    \prod_{k=1}^K
    \left(
      \frac{n_k}{N}
    \right)^{n_k},
  \]
  and, therefore,
  \begin{multline}\label{eq:core}
    \URE(F)
    =
    \bar F(z_1,\dots,z_N)
    \frac{N!}{n_1!\dots n_K!}
    \prod_{k=1}^K
    \left(
      \frac{n_k}{N}
    \right)^{n_k}\\
    =
    \UXE(F)
    \frac{N!}{n_1!\dots n_K!}
    \prod_{k=1}^K
    \left(
      \frac{n_k}{N}
    \right)^{n_k}
    \ge
    \UXE(F)
    \frac{N!}{N^N},
  \end{multline}
  where the inequality follows from $n^n/n!\ge1$.
  This completes the proof of \eqref{eq:to-prove}.

  The example in the previous section demonstrates
  that the bound is tight when $\left|\mathbf{Z}\right|\ge N$,
  since we may assume $\{1,\dots,N\}\subseteq\mathbf{Z}$
  without further loss of generality.
\end{proof}

\subsection{Smaller observation spaces $\mathbf{Z}$}
\label{subsec:finite-Z}

The case $\left|\mathbf{Z}\right|=\infty$,
or at least $\left|\mathbf{Z}\right|\gg1$,
is probably most relevant in practice
(e.g., even in the case of classification problems,
the objects to be classified are typically complex).
In this subsection, however, we consider the case of a finite $\mathbf{Z}$
and are mostly interested in a small $\left|\mathbf{Z}\right|$.
This will allow us to cover, e.g.,
Kolmogorov's and Martin-L\"of's case of binary sequences.

The inequality \eqref{eq:to-prove} in Proposition~\ref{prop:infinite-Z}
is tight in the case of an infinite $\mathbf{Z}$,
which we relaxed in the statement of the proposition
to what is actually used in the proof, $\left|\mathbf{Z}\right|\ge N$.
The following proposition extends \eqref{eq:to-prove} to smaller $\mathbf{Z}$.

\begin{proposition}\label{prop:finite-Z}
  Suppose $K:=\left|\mathbf{Z}\right|\le N$.
  For any random variable $F:\mathbf{Z}^N\to[0,\infty)$,
  we have $\UXE(F) \le C\URE(F)$,
  where
  \begin{equation}\label{eq:C}
    C
    :=
    \frac{N^N}{N!}
    \prod_{k=1}^K
    \frac{n_k!}{n_k^{n_k}}
  \end{equation}
  and $n_k$ is any balanced split of $N$ into $K$ parts:
  \begin{equation}\label{eq:split}
    n_k
    \in
    \left\{
      \lfloor N/K\rfloor,
      \lceil N/K\rceil
    \right\}
    \text{ such that }
    \sum_{k=1}^K n_k = N.
  \end{equation}
  The factor $C$ is tight;
  moreover,
  there exists an event $A\subseteq\mathbf{Z}^N$
  such that
  \begin{equation}\label{eq:tight}
    \URP(A)
    =
    1/C
    \le
    1 = \UXP(A).
  \end{equation}
\end{proposition}

\begin{proof}
  Without loss of generality,
  we set $\mathbf{Z}:=\{1,\dots,K\}$ and proceed
  as in the proof of Proposition~\ref{prop:infinite-Z}.
  Let $n_k$ be the number of times that $k\in\mathbf{Z}$
  occurs in $z_1,\dots,z_N$
  (so that now $n_k=0$ is possible,
  in which case $n_k^{n_k}:=1$ and $n_k!=1$).
  Now instead of the inequality ``$\ge$'' in \eqref{eq:core}
  we use the convexity of the function
  $\log(n^n/n!) = n\log n - \log(n!)$ in $n$
  (see Lemma~\ref{lem:convexity} in Appendix~\ref{app:A}).
  If the split $n_k$, $k=1,\dots,K$, of $N$ is not balanced,
  we can find $n_{k_1}$ and $n_{k_2}$ such that $n_{k_1}-n_{k_2}\ge2$.
  By the convexity, the penultimate expression in the chain \eqref{eq:core}
  cannot increase if we move $n_{k_1}$ and $n_{k_2}$ towards each other
  by redefining $n_{k_1}:=n_{k_1}-1$ and $n_{k_2}:=n_{k_2}+1$.
  Repeating this operation we arrive at a balanced split.

  An event $A$ satisfying \eqref{eq:tight} can be chosen as the orbit
  with counts $n_k$ for $k\in\mathbf{Z}$ satisfying \eqref{eq:split}.
\end{proof}

Of course, the inequalities in Propositions~\ref{prop:infinite-Z} and~\ref{prop:finite-Z}
remain true if the range $[0,\infty)$ of $F$ is replaced by $[0,\infty]$
(as in the definition of e-variables).

\begin{figure}[bt]
  \begin{center}
    \includegraphics[width=0.98\columnwidth]{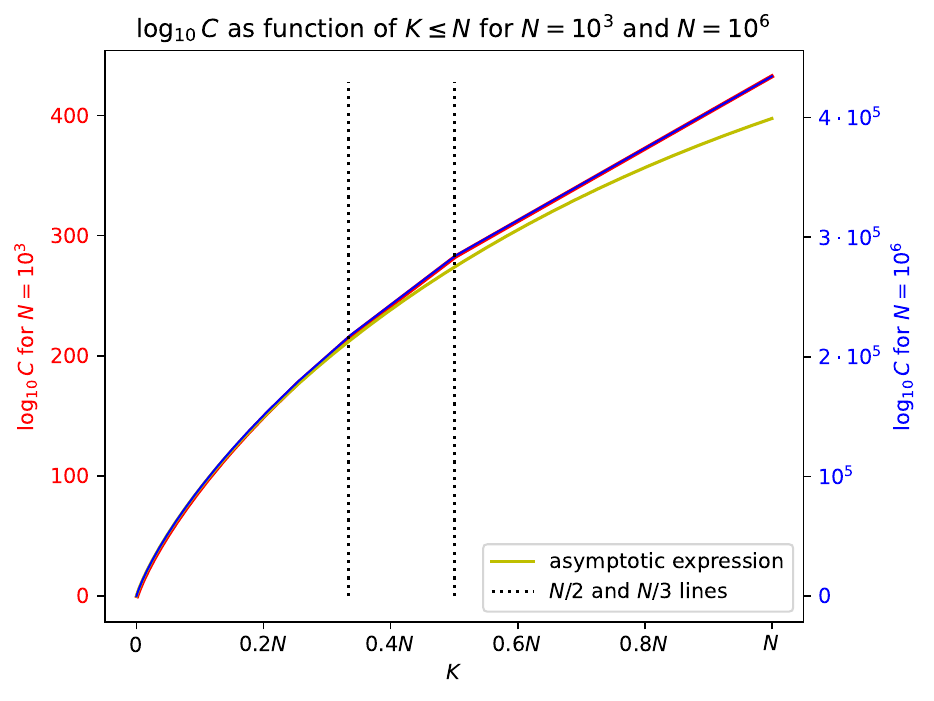}
  \end{center}
  \caption{The graphs of $\log_{10}C$ for $N=10^3$
    (thick red line, with some values given on the left)
    and $N=10^6$ (thin blue line, with values on the right).
    The yellow line represents the approximation
    described in Remark~\ref{rem:limit}.}
  \label{fig:logC_fixedN}
\end{figure}

Figure~\ref{fig:logC_fixedN} shows $\log_{10} C$
as a function of $K\in\{1,\dots,N\}$ for two values of $N$,
$10^3$ (red, with the scale of $\log_{10}C$ on the left)
and $10^6$ (blue, with the scale on the right).
Both graphs become horizontal to the right of $K:=N$
(not shown in Fig.~\ref{fig:logC_fixedN}).
We can see that the shapes of the two graphs are very similar.
The blue line (for $10^6$) was drawn after the red one (for $10^3$),
and the latter is thicker in order to be able to see the small difference between the graphs.
When implementing the formula \eqref{eq:C} for $C$,
it is convenient to compute the number of $n_k=\lceil N/K\rceil$ in \eqref{eq:split}
as $N \bmod K$.

The vertical dotted lines in Fig.~\ref{fig:logC_fixedN}
are drawn through the observation $N/2$ (the right line)
and the nearest observation to $N/3$ (the left one).
The qualitative behaviour of the red and blue lines,
namely the strictly exponential growth
(linear on the log scale of Fig.~\ref{fig:logC_fixedN})
of both graphs between, roughly,
$N/2$ and $N$, $N/3$ and $N/2$, etc., is easy to understand.
When $K$ increases by 1 between $N/2$ and $N$,
one of the $n_k=2$ in \eqref{eq:split}
gets replaced by two $n_k$, namely 1 and 1.
Therefore, the expression for $C$ given by \eqref{eq:C}
gets multiplied by $2^2/2!=2$,
and so the slope of the red and blue lines
between $N/2$ and $N$ in Fig.~\ref{fig:logC_fixedN}
is $\log_{10}2\approx0.301$.
When $K$ increases by 1 between $N/3$ and $N/2$,
a block $(3,3)$ of two $n_k$ in \eqref{eq:split}
gets replaced by a block $(2,2,2)$ of three $n_k$.
The expression for $C$ given by \eqref{eq:C} gets multiplied by
\[
  \left(
    \frac{2!}{2^2}
  \right)^{3}
  /
  \left(
    \frac{3!}{3^3}
  \right)^2
  =
  3^4/2^5,
\]
and so the slope of the red and blue lines
between $N/3$ and $N/2$ in Fig.~\ref{fig:logC_fixedN}
is $\log_{10}(3^4/2^5)\approx0.403$.
And when $K$ increases by 1 between $N/4$ and $N/3$,
a block $(4,4,4)$ of three $n_k$ gets replaced by a block $(3,3,3,3)$ of four $n_k$.
The expression for $C$ gets multiplied by
\[
  \left(
    \frac{3!}{3^3}
  \right)^{4}
  /
  \left(
    \frac{4!}{4^4}
  \right)^3
  =
  2^{19}/3^{11},
\]
and so the slope of the red and blue lines
immediately to the left of $N/3$ in Fig.~\ref{fig:logC_fixedN}
is $\log_{10}(2^{19}/3^{11})\approx0.471$.
The slope keeps increasing as we move left.

\begin{remark}\label{rem:limit}
  Let us replace the decimal logarithms $\log_{10}$ by natural $\ln$
  in the graphs shown in Fig.~\ref{fig:logC_fixedN} for $N\in\{10^3,10^6\}$.
  This will not change the shape of the graphs and will only change
  the labels on the axes;
  the upper limit of the range of $\log C$ will now become close to $N$,
  which follows from Stirling's formula
  applied to the special cases of Propositions~\ref{prop:infinite-Z} and~\ref{prop:finite-Z}
  corresponding to $K=N$:
  \begin{equation}\label{eq:L1}
    \ln\frac{N^N}{N!}
    \sim
    N.
  \end{equation}
  An interesting function is the limit $L$ of these graphs as $N\to\infty$
  with both axes rescaled by dividing by $N$
  (so that the slopes remain unchanged).
  It can be defined as the continuous piecewise-linear function $L:[0,1]\to[0,\infty)$
  satisfying $L(0):=0$ and
  \begin{equation}\label{eq:L'}
    L'(x)
    :=
    \ln
    \left(
      \left(
        \frac{n!}{n^n}
      \right)^{n+1}
      \left(
        \frac{(n+1)!}{(n+1)^{n+1}}
      \right)^{-n}
    \right)
  \end{equation}
  for all
  $
    x
    \in
    \left(
      \frac{1}{n+1},
      \frac{1}{n}
    \right)
  $
  and all $n\in\{1,2,\dots\}$.
  According to \eqref{eq:L1}, $L(1)=1$,
  and since \eqref{eq:C} gives $C=1$ for $K=1$, $L(0)=0$.
  As $x\to0$,
  \begin{equation}\label{eq:L}
    L(x)
    =
    -\frac12 x\ln x
    + x\ln\sqrt{2\pi}
    + o(x)
  \end{equation}
  (see Appendix~\ref{app:C} for derivation details).
  The right-hand side of \eqref{eq:L} without the ``${}+o(x)$''
  and with $\log_{10}$ in place of $\ln$
  is shown as the yellow line in Fig.~\ref{fig:logC_fixedN}.
  The final value of the approximation
  $
    -\frac12 x\ln x
    + x\ln\sqrt{2\pi}
  $
  at $x=1$ is approximately 0.919,
  which is not so different from $L(1)=1$.
  We can see from \eqref{eq:L} that the slope of $L(x)$ is infinity at $x=0$.
\end{remark}

Another interesting case for a finite $\mathbf{Z}$
is where  $K:=\left|\mathbf{Z}\right|$ is fixed
while the number of observations $N$ varies.
Applying Stirling's formula to \eqref{eq:C}
we then obtain
\begin{equation}\label{eq:CC}
  C
  =
  \Theta(N^{(K-1)/2}).
\end{equation}
This polynomial growth rate as $N\to\infty$
contrasts with the exponential growth rate for $\left|\mathbf{Z}\right|=\infty$.
The closeness of $\UXE$ and $\URE$ to within a polynomial factor
(namely, $\Theta(N^{(K-1)/2})$)
implies the closeness of e-variables
under exchangeability and randomness in the same crude sense;
it also implies the closeness of $\UXP$ and $\URP$,
which in turn implies the closeness of p-variables
under exchangeability and randomness, in the same sense.

On Kolmogorov's and Martin-L\"of's log scale, $N^{(K-1)/2}$ becomes $\frac{K-1}{2}\log N$,
and differences of $O(\log N)$ are often ignored in the algorithmic theory of randomness;
according to Kolmogorov,
``we should not be afraid of logarithms (as well as $O(1)$ terms that we have anyway)''.%
\footnote{In Russian, ``\begin{cyr}логарифмов не надо бояться, так же как и констант\end{cyr}'';
  recorded by Shen \cite[note 12]{Semenov/etal:2024-full}
  and translated by the authors of \cite[arXiv version]{Semenov/etal:2024-full}.}
In this sense exchangeability and being IID nearly coincide for finite sequences as well.
However, for $\left|\mathbf{Z}\right|=\infty$ the difference becomes $\Theta(N)$
on the logarithmic scale,
of the order of the largest possible values for $m$ and $\mR$
in Kolmogorov's and Martin-L\"of's binary setting.

\begin{remark}\label{rem:connection}
  The expression $\frac{K-1}{2}\log N$ is akin to many standard expressions
  for the complexity of a regular statistical model with $K-1$ free parameters
  for $N$ observations; see, e.g., \cite[Theorem~7.1]{Grunwald:2007}.
  Remember that the number of free parameters in $\mathfrak{P}(\mathbf{Z})$
  for $\left|\mathbf{Z}\right|=K$ is $K-1$ (the probabilities of $z\in\mathbf{Z}$ sum to 1).
  For $K=2$, the expression $\frac{K-1}{2}\log N$ follows
  from the much more precise results of \cite{Vovk:1986-proofs},
  and the argument given there generalizes to $K\ge2$.
\end{remark}

More recently the log scale for p-values has been advocated by Greenland,
who called $-\log_2 p$ the \emph{S-value}, or \emph{p-surprisal},
corresponding to a p-value of $p$;
see, e.g., \cite{Greenland:2024}.
There have been no suggestions to ignore logarithms in p-surprisals,
which were designed to be closer to statistical practice.
Indeed, from the practical point of view logarithms are important
even for binary sequences:
e.g., if we toss a coin (possibly biased) $N:=10^3$ times and get exactly half ``heads'',
we will have statistically significant evidence that the tosses are not IID,
since the largest product probability $Q^N$ of this very simple event is about $2.52\%$;
on the other hand, the sequence of outcomes may be perfectly exchangeable.
The value of the largest probability can be computed from
\begin{equation}\label{eq:binary}
  \URP(A)
  =
  \frac{N!}{(N/2)!^2}
  2^{-N}
  \le
  1
  =
  \UXP(A),
\end{equation}
where $A$ is the event of observing $N/2$ ``heads'' (assuming $N$ even).

\begin{remark}
  The case of a fixed $K:=\left|\mathbf{Z}\right|$ with variable $N$
  corresponds to the bottom left corner of Fig.~\ref{fig:logC_fixedN}.
  According to \eqref{eq:CC} (and the end of Remark~\ref{rem:limit}),
  the slope of the graph of $\log_{10}C$ in that corner
  becomes infinite ``under microscope'' as $N\to\infty$;
  namely, we expect the slope to grow as $\frac12\log_{10}N$.
  For $N=10^3$ and $N=10^6$, as used in Fig.~\ref{fig:logC_fixedN},
  this expression gives the slopes of 1.5 and 3, respectively;
  more precise values given by \eqref{eq:binary}
  are 1.598 and 3.098, respectively.
\end{remark}

\begin{figure*}[bt]
  \begin{center}
    \includegraphics[width=0.50\columnwidth]{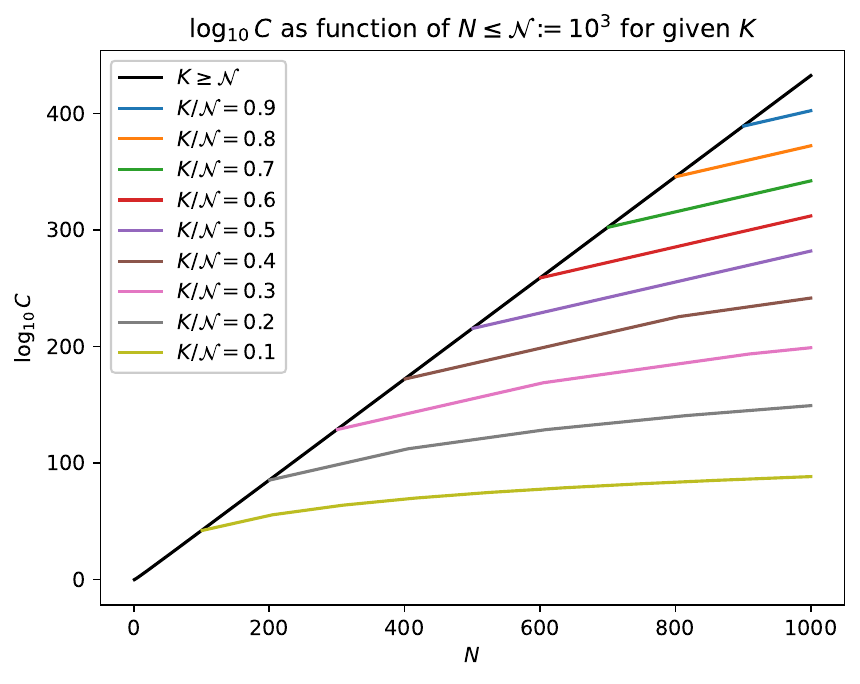}
    \includegraphics[width=0.48\columnwidth]{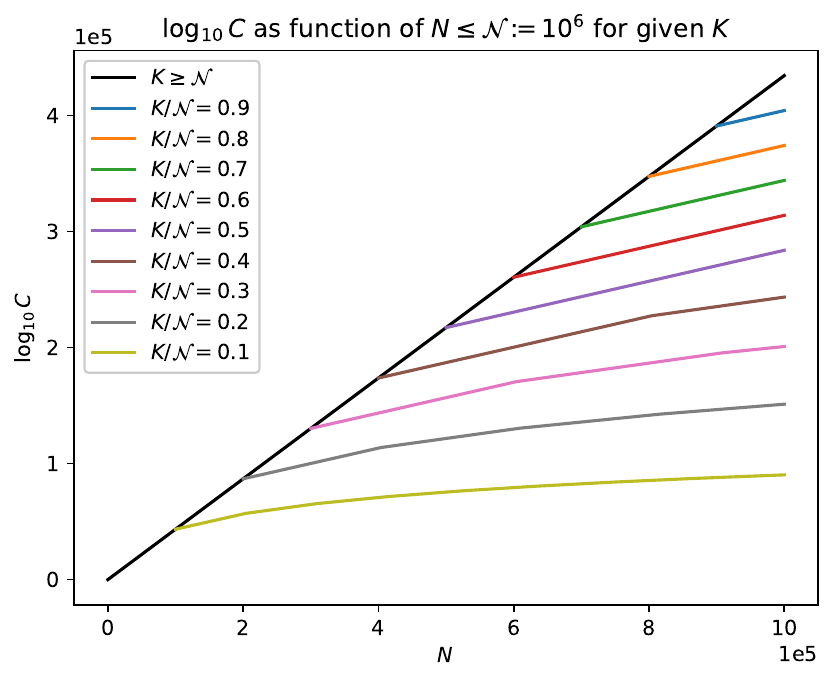}
  \end{center}
  \caption{The graphs of $\log_{10}C$
    for $\mathcal{N}=10^3$ (left panel) and $\mathcal{N}=10^6$
    (right panel, using $10^5$ as the unit for the labels on both axes),
    as described in the text.}
  \label{fig:logC_fixedK}
\end{figure*}

Figure~\ref{fig:logC_fixedK} complements Fig.~\ref{fig:logC_fixedN}
by plotting $\log_{10}C$ as function of $N$ for a fixed $K$ explicitly.
It shows two ranges for $N$, $N\le10^3$ in the left panel and $N\le10^6$ in the right panel;
the two panels look similar, and the description to follow is applicable to either.
Let $\NNN\in\{10^3,10^6\}$ be the upper limit of the range of $N$.
The base graph shows $\log_{10}C$ in black as function of $N$ for $K\ge\NNN$.
The remaining 9 graphs consist of two pieces, black and coloured;
the graphs and their labels in the legends are shown in the same order,
top to bottom.
Let me describe, for concreteness, the bottom one labelled as $K/\NNN=0.1$;
this description will also be applicable, \emph{mutatis mutandis},
to the other 8 graphs.
The graph corresponds to $K=0.1\,\NNN$ and consists of two parts:
the values for $N\le K$ are shown in black and the values for $N>K$ in olive.
The behaviour of the graph changes drastically after $N=K$,
which is marked by using a different colour:
the black part grows exponentially fast,
while the olive part grows only polynomially fast
(albeit for a polynomial of a high degree,
namely $\lceil(K-1)/2\rceil\ge50$ according to \eqref{eq:CC}).

Let us see why the graphs in Fig.~\ref{fig:logC_fixedK} look the way they do.
The black graph is the diagonal of the bounding rectangle in the limit $\NNN\to\infty$.
Suppose $K=\frac{k}{10}\,\NNN$, $k\in\{1,\dots,9\}$,
is one of the $K$ marked in Fig.~\ref{fig:logC_fixedK}.
If $nK<N<(n+1)K$ for an integer $n\ge1$, all $n_k$ in a balanced split \eqref{eq:split}
are either $n$ or $n+1$.
Incrementing $N$ by 1 to $N+1$ leads to replacing one of the $n_k=n$ by $n_k=n+1$.
The relative increment in the constant $C$ given by \eqref{eq:C} is
\begin{multline}\label{eq:slope}
  \left(
    \frac{(N+1)^{N+1}}{(N+1)!}
    \frac{(n+1)!}{(n+1)^{n+1}}
  \right)
  /
  \left(
    \frac{N^N}{N!}
    \frac{n!}{n^{n}}
  \right)\\
  =
  \left(
    1 + \frac1N
  \right)^N
  \left(
    1 + \frac1n
  \right)^{-n}
  \approx
  \e
  \left(
    1 + \frac1n
  \right)^{-n}
\end{multline}
(the ``$\approx$'' is justified by $N>K$
and all the $K$ marked in Fig.~\ref{fig:logC_fixedK}
being large, at least 100).

For $n=1$, the last expression in \eqref{eq:slope}
gives the slope of
$\log_{10}\e-\log_{10}2\approx0.133$
between $K$ and $2K$.
This is the slope of the full coloured lines for $K=0.9\,\mathcal{N}$ (blue)
to $K=0.5\,\mathcal{N}$ (purple)
and the slope of the first straight segment of the other coloured lines
(strictly speaking, ``straight'' should be understood as ``approximately straight''
because of the ``$\approx$'' in \eqref{eq:slope}).
The slope of the following straight segment of the coloured lines
for $K=0.4\,\mathcal{N}$ (brown) to $K=0.1\,\mathcal{N}$ (olive)
is $\log_{10}\e-2\log_{10}1.5\approx0.082$.
The slope of the following straight segment
for the bottom three coloured lines is
$\log_{10}\e-3\log_{10}(4/3)\approx0.059$, etc.
It is clear from \eqref{eq:slope} that the slope tends to 0
when $n\to\infty$,
and we can see that the bottom line of Fig.~\ref{fig:logC_fixedK}
is close to being horizontal on the right.
If any of the coloured lines is continued to the right beyond $\mathcal{N}$,
it will consist of segments of exponentially fast growth with decreasing growth rates,
which will make the overall growth rate polynomial.

\section{Conclusion}
\label{sec:conclusion}

Being motivated by the foundations of probability and statistics,
de Finetti, Kolmogorov, and Martin-L\"of considered cases
where the assumptions of exchangeability and randomness
are close to each other.
However, there are also cases where the closeness of the two assumptions disappears,
including the important case of finite sequences of a given length
for a large observation space.
Interestingly,
this difference has only a limited effect on the efficiency of conformal prediction.

De Finetti's theorem has many fascinating generalizations and variations,
and we can ask similar questions about those.
One generalization that is especially close to the subject of this paper
concerns weighted exchangeability \cite{Barber/etal:2024},
which accounts for a known covariate shift.
Many more are provided by the theory of repetitive structures;
see, e.g., \cite{Lauritzen:1988}, \cite[Part~IV]{Vovk/etal:2022book},
and \cite[Chap.~4]{Bernardo/Smith:2000}
(the last book, however, does not use the terminology of repetitive structures).

\subsection*{Acknowledgments}

Thanks to the participants of the workshop ``Algorithmic Statistics''
(Oxford, November 28, 2025)
for a useful discussion
and to Ioannis Kontoyiannis for informing me about \cite{Johnson/etal:2025}
and other related work.

I acknowledge the use of Microsoft Copilot in exploring proof ideas,
which I reviewed carefully.
The final version was checked using Claude Opus 5.
I take full responsibility for this paper's mathematical statements
and their proofs.

Comments by the anonymous reviewers of the journal version of this paper
were very helpful in improving the presentation;
in particular, they inspired Remarks~\ref{rem:anonymous} and~\ref{rem:connection}.

\appendix
\section{A proof of convexity}
\label{app:A}

The following lemma was used in the proof of Proposition~\ref{prop:finite-Z}.

\begin{lemma}\label{lem:convexity}
  The function $n\log n - \log(n!)$ of $n\in\{0,1,\dots\}$
  is convex (and even strictly convex).
\end{lemma}

\begin{proof}
  We are required to prove that the function
  \begin{multline*}
    f(n)
    :=
    \left(
      (n+1)\log(n+1) - \log((n+1)!)
    \right)\\
    -
    \left(
      n\log n - \log(n!)
    \right)
    =
    n
    \log
    \left(
      1 + 1/n
    \right)
  \end{multline*}
  (with $f(0)$ understood to be $0$)
  is strictly increasing.
  Extending it to the nonnegative reals, 
  \(
    f(x)
    :=
    x\log(1 + 1/x)
  \),
  interpreting $\log$ as $\ln$, and differentiating gives
  \[
    f'(x)
    =
    \log
    \left(
      1+\frac{1}{x}
    \right)
    -
    \frac{1}{x+1}.
  \]
  By the strict concavity of \(\log\),
  we have \(\log(1+u)>\frac{u}{1+u}\) for \(u>0\)
  (this inequality compares the increment of $\log$ over the interval $[1,1+u]$
  with the increment of its tangent at $1+u$),
  and substituting \(u:=1/x\) gives $f'(x)>0$ for all \(x>0\).
\end{proof}

\section{An asymptotic expression for $L$}
\label{app:C}

In this appendix we check the asymptotic expression \eqref{eq:L}
for the function $L$ given in Remark~\ref{rem:limit}.
It is sufficient to consider only $x$ of the form $1/N$,
$N$ being a positive integer
(since the asymptotic expression \eqref{eq:L} for $L(x)$ with $x$ of this form implies
the same asymptotic expression for $L(x+O(x^2))$).

We can simplify the expression for $L'(x)$ in \eqref{eq:L'} as
\begin{equation*}
  L'(x)
  =
  \ln(n!)
  +
  n(n+1)\ln\frac{n+1}{n}
  -
  n\ln(n+1),
\end{equation*}
which implies, by Stirling's formula,
\[
  L'(x)
  =
  \frac12 \ln n
  + \ln\sqrt{2\pi}
  - \frac12
  + O(1/n).
\]
Therefore, the proof is completed by noticing that
\begin{align*}
  L\left(\frac1N\right)
  &=
  \sum_{n=N}^{\infty}
  \frac{
    \frac12 \ln n
    + \ln\sqrt{2\pi}
    - \frac12}
  {n(n+1)}
  +
  O(N^{-2})\\
  &=
  \frac12
  \frac{\ln N}{N}
  +
  \frac{\ln\sqrt{2\pi}}{N}
  +
  o(N^{-1}).
\end{align*}
The last equality uses
\begin{equation}\label{eq:ln_N}
  \sum_{n=N}^{\infty}
  \frac{\ln n}{n(n+1)}
  =
  \frac{\ln N+1}{N}
  +
  o(N^{-1})
\end{equation}
and
\[
  \sum_{n=N}^{\infty}
  \frac{1}{n(n+1)}
  =
  \sum_{n=N}^{\infty}
  \left(
    \frac{1}{n}
    -
    \frac{1}{n+1}
  \right)
  =
  \frac{1}{N}.
\]
Let us check \eqref{eq:ln_N} in the form
\begin{equation*}
  \sum_{n=N}^{\infty}
  \frac{\ln n}{n^2}
  =
  \frac{\ln N+1}{N}
  +
  o(N^{-1}).
\end{equation*}
This follows from the integral approximation
\[
  \int_N^{\infty}
  \frac{\ln x}{x^2}
  \d x
  =
  \left[
    -\frac{\ln x+1}{x}
  \right]_N^{\infty}
  =
  \frac{\ln N+1}{N}.
\]
\end{document}